\documentclass[a4paper,reqno]{amsart} %
\usepackage{hyperref} 
\renewcommand{\a}{\alpha}
\renewcommand{\b}{\beta}
\newcommand{\V}{\mathcal{V}}          
\renewcommand{\H}{\mathbf H}          
\newcommand{\N}{\mathbb N}            

\newcommand{\SSS}{\mathbf{s}}

\DeclareMathOperator{\ad}{ad}
\DeclareMathOperator{\ch}{char}

\DeclareMathOperator{\St}{St}              
\DeclareMathOperator{\gr}{gr}

\newtheorem{Th}{Theorem}[section]
\newtheorem{Lemma}[Th]{Lemma}

\begin{document}
\title{Identities in group rings, enveloping algebras and Poisson algebras}
\author{Victor Petrogradsky}
\address{Department of Mathematics, University of Brasilia, 70910-900 Brasilia DF, Brazil}
\email{petrogradsky@rambler.ru}
\thanks{The author was partially supported by grant
CNPq~309542/2016-2 
\\
for proceedings of Workshop INDAM
{\it Polynomial Identities in Algebras}, 2019 Roma, Italy}

\dedicatory{Dedicated to 70th anniversary of Antonio Giambruno}

\begin{abstract}
This is a short survey of works on identical relations in group rings,
enveloping algebras, Poisson symmetric algebras and other related algebraic structures.
First, the classical work of Passman specified group rings that satisfy nontrivial identical relations.
This result was an origin and motivation of close research projects.
Second, Latyshev and Bahturin determined Lie algebras such
that their universal enveloping algebra satisfies a non-trivial identical relation.
Next, Passman and Petrogradsky solved a similar problem in case  of restricted enveloping algebras.
Third, Farkas started to study identical relations in Poisson algebras.
On the other hand,
Shestakov proved that the symmetric algebra $S(L)$ of an arbitrary Lie algebra $L$ satisfies the identity
$\{x,\{y,z\}\}\equiv 0$ if, and only if, $L$ is abelian.
We survey further results on existence of identical relations in
(truncated) Poisson symmetric algebras of Lie algebras.
In particular, we report on recent results on (strong) Lie nilpotency
and (strong) solvability of (truncated) Poisson symmetric algebras and related nilpotency classes.
Also, we discuss constructions and methods to achieve these results.
\end{abstract}

\maketitle
\section{Introduction}
Now, there is an established theory of identical relations
in associative and Lie algebras~\cite{Drensky,Ba}.
It has many applications to group theory such as the solution of the Restricted Burnside Problem.
Also, identical relations were applied to study other algebraic structures.

In the present review,
we discuss existence of identical relations in three classes of algebras.
First,  the starting point is the result of Passman on existence
of identical relations in group rings~\cite{Pas72} (Theorem~\ref{TKG}).
This result caused an intensive research on different types of identical relations in group rings,
such as Lie nilpotence, solvability, non-matrix identical relations,
classes of Lie nilpotence, solvability lengths, etc.
There are at least 50 papers published in this area.

Second, Latyshev~\cite{Lat63} and Bahturin~\cite{Ba74}
started to study identical relations in universal enveloping algebras of Lie algebras.
Passman~\cite{Pas90} and Petrogradsky~\cite{Pe91} specified
existence of identical relations in restricted enveloping algebras (Theorem~\ref{TuL}).
There are many papers in this area studying different types of identical relations,
such as Lie nilpotence, solvability, non-matrix identical relations, classes of Lie nilpotence, solvability lengths, etc.
In particular, Riley and Shalev determined necessary and sufficient conditions for restricted Lie algebras under which
the restricted enveloping algebra is Lie nilpotent or solvable~\cite{RiSh93}.
The research was further extended to new objects, such as Lie superalgebras, color Lie superalgebras, smash products.
These problems were studied in numerous
papers by Bahturin, Bergen, Kochetov, Lichtman, Passman, Petrogradsky, Riley, Shalev, Siciliano, Spinelli, Usefi, et.al.

Poisson algebras appeared in works of Berezin~\cite{Ber67} and Vergne~\cite{Ver69}.
Free Poisson (super)algebras were introduced by Shestakov~\cite{Shestakov93}.
A basic theory of identical relations for Poisson algebras  was developed by Farkas~\cite{Farkas98,Farkas99}.
Identical relations of symmetric Poisson algebras
of Lie (super)algebras were studied by Kostant~\cite{Kos81}, Shestakov~\cite{Shestakov93}, and Farkas~\cite{Farkas99}.
The third starting point for our research is the result of Giambruno and Petrogradsky~\cite{GiPe06}
on existence of non-trivial multilinear Poisson identical relations in truncated symmetric Poisson algebras
of restricted Lie algebras (Theorem~\ref{Tmain}).
Finally, we review recent results on Lie identities of truncated symmetric Poisson algebras~\cite{PeIlana17}.
\medskip

By $K$ denote the ground field, as a rule of positive characteristic $p$.
By $\langle S\rangle$ or $\langle S\rangle_K$ denote the linear span of a subset $S$ in a $K$-vector space.
Let $L$ be a Lie algebra.
The Lie brackets are left-normed: $[x_1,\ldots,x_n]=[[x_1,\dots,x_{n-1}],x_n]$, $n\ge 1$.
Adjoint maps $\ad x:L\to L$ are defined by $\ad x(y)=[x,y]$, $x,y\in L$.
One defines the {\it lower central series}: $\gamma_1(L)=L$, $\gamma_{n+1}(L)=[\gamma_{n}(L), L]$, $n\ge 1$.
Also, $L^2=[L,L]=\gamma_2(L)$ is the {\it commutator subalgebra}.
By $U(L)$ denote the universal enveloping algebra and $S(L)=\mathop{\oplus}\limits_{n=0}^\infty U_n/U_{n-1}$
the related {\it symmetric algebra}~\cite{Ba,BMPZ,Dixmier}.
For the basic theory of restricted Lie algebras and restricted enveloping algebras see~\cite{Ba,JacLie}.
Let us note that all our Lie algebras over a field of positive characteristic need not be restricted.

\section{Identical relations of group rings}
\label{Sidentities}

In this section we review results on existence of nontrivial polynomial identities in group rings.
This is the origin of this research direction.

Passman obtained necessary and sufficient conditions for a group ring $K[G]$ to
satisfy a nontrivial polynomial identity over a field $K$ of arbitrary characteristic $p$.
A group $G$ is said to be {\it p-abelian} if $G$ is abelian in case $p=0$ and, in case $p>0$, $G'$,
the commutator subgroup of $G$, is a finite $p$-group.

\begin{Th}[\cite{Pas72}]
\label{TKG}
The group  algebra $K[G]$ of a group $G$ over a field $K$ of characteristic $p\ge 0$ satisfies
a nontrivial polynomial identity if and only if the following conditions are satisfied.
\begin{enumerate}
  \item
    There exists a subgroup  $A\subseteq G$ of finite index;
  \item
    $A$ is $p$-abelian.
\end{enumerate}
\end{Th}

All our associative algebras are with unity.
Recall the notions of a (strong) Lie nilpotence and (strong) solvability for associative algebras.
Let $A$ be an associative algebra, and $A^{(-)}$ the related Lie algebra.
Consider its {\it lower central series}:
$\gamma_1(A)=A$, $\gamma_{i+1}(A)=[\gamma_i(A),A]$, $i\ge 1$.
The algebra $A$ is said to be {\it Lie nilpotent} of class $s$ if and only if  $\gamma_{s+1}(A)=0$ and $\gamma_s(A)\ne 0$.
Also consider {\it upper Lie powers} defined by $A^{(0)}=A$ and $A^{(n+1)}=[A^{(n)},A]A$, $n\ge 0$
(we use the shifted enumeration in comparison with~\cite{PaPaSe73,RiSh95} because one checks that
$\{A^{(n)}\,|\, n\ge 0\}$ is a filtration).
Now, $A$ is {\it strongly Lie nilpotent} of class $s$ if and only if $A^{(s)}=0$ and $A^{(s-1)}\ne 0$.
One defines the {\it derived series} of $A$ by setting $\delta_{0}(A)=A$, $\delta_{i+1}(A)=[\delta_{i}(A),\delta_{i}(A)]$, $i\ge 0$.
The algebra $A$ is {\it solvable} of length $s$ if and only if $\delta_{s}(A)=0$ and $\delta_{s-1}(A)\ne 0$.
Consider also the {\it upper derived series}: $\tilde\delta_{0}(A)=A$,
$\tilde\delta_{i+1}(A)=[\tilde\delta_{i}(A),\tilde\delta_{i}(A)]A$, $i\ge 0$.
Now, $A$ is {\it strongly solvable} of length $s$ if and only if
$\tilde\delta_{s}(A)=0$ and $\tilde\delta_{s-1}(A)\ne 0$.

Passi, Passman and Sehgal characterized the Lie nilpotence and solvability of $K[G]$~\cite{PaPaSe73}.
\begin{Th}[\cite{PaPaSe73}]\label{TKGNS}
Let $K[G]$ be the group ring of a group $G$ over a field $K$, $\ch K=p\geq 0$. Then
\begin{enumerate}
  \item
    $K[G]$ is Lie nilpotent if and only if $G$ is $p$-abelian and nilpotent;
  \item
    $K[G]$ is solvable if and only if $G$ is $p$-abelian, for $p\ne 2$;
  \item
    $K[G]$ is solvable if and only if $G$ has a $2$-abelian subgroup of index at most $2$, for $p=2$.
\end{enumerate}
\end{Th}

Using the upper Lie powers, one defines {\it Lie dimension subgroups} of a group (our enumeration is shifted)~\cite{Passi}:
$$
D_{(n),K}(G)=G\cap (1+K[G]^{(n)}), \qquad n\ge 0.
$$
One also has the following description~\cite{BhaPas92}:
\begin{equation}\label{dim-subgroups}
D_{(n),K}(G)=\prod_{(i-1)p^k\ge n} \gamma_i(G)^{p^k},\quad n\ge 0.
\end{equation}
There is a formula for the Lie nilpotency class of a modular group ring.
\begin{Th}[\cite{BhaPas92}]\label{Tgr-nilp}
Let $G$ be a group, $K$ a field of characteristic $p>3$ such
that the group ring $K[G]$ is Lie nilpotent.
Then the Lie nilpotency class of $K[G]$ coincides with its strong Lie nilpotency class and is equal to
$$
1+(p-1)\sum_{m\ge 1}m \log_p \big|D_{(m),K}(G):D_{(m+1),K}(G)\big|.
$$
\end{Th}
The original work  cased more research projects.
So called generalized polynomial identities of (twisted) group rings were studied in~\cite{Pas71,Pas01}.
Also, a relation with gradings of PI-algebras see in~\cite{AljDav15}.

A group $G$ is said to have the $n$-{\it rewritable property} $Q_n$ if for all
elements $g_1,g_2,\ldots,g_n \in G$, there exist two distinct permutations $\sigma,\tau\in Sym_ n$ such that
$g_{\sigma(1)}g_{\sigma(2)}\cdots g_{\sigma(n)} = g_{\tau(1)}g_{\tau(2)}\cdots g_{\tau(n)}$~\cite{Blyth88}.
The following result is a further application of delta-sets (see definitions below).
\begin{Th}[\cite{Blyth88,ElaPas11}]
If a group $G$ satisfies $Q_n$, then $G$ has a characteristic
subgroup $N$ such that $| G : N |$ and $|N'|$ are both finite and have
sizes bounded by functions of $n$.
\end{Th}

\subsection*{Problem}
We suggest a  problem to find an analogue of the rewritable property and characterise it  in case of Lie algebras.

\section{Identical relations of enveloping algebras}
Latyshev proved that the universal enveloping algebra of a finite
dimensional Lie algebra over a field of characteristic zero satisfies a nontrivial
polynomial identity if and only if the Lie algebra is abelian~\cite{Lat63}.
Bahturin noticed that the condition of a finite dimensionality is inessential (see e.g.~\cite{Ba}).

Bahturin settled a similar problem on the existence of  a
nontrivial identity for the universal enveloping algebra over a
field of positive characteristic~\cite{Ba74}.
Also, Bahturin found necessary and sufficient conditions for the universal enveloping algebra
of a Lie superalgebra over a field of characteristic zero to satisfy a non-trivial polynomial identity~\cite{Ba85}.
PI-subrings and algebraic elements in universal  enveloping algebras and their fields of fractions were studied by Lichtman~\cite{Lichtman89}.
\medskip

Passman~\cite{Pas90} and Petrogradsky~\cite{Pe91} described
restricted Lie algebras $L$ whose restricted enveloping
algebra $u(L)$ satisfies a nontrivial polynomial identity.

\begin{Th}[\cite{Pas90},~\cite{Pe91}]
\label{TuL}
Let $L$ be a restricted Lie algebra over a field of characteristic $p>0$.
The restricted enveloping algebra $u(L)$ satisfies
a nontrivial polynomial identity if and only if
there exist restricted ideals $Q\subseteq H\subseteq L$ such that
 \begin{enumerate}
   \item
     $\dim L/H<\infty$, $\dim Q<\infty$;
   \item
     $H/Q$ is abelian;
   \item
     $Q$ is abelian and has a nilpotent $p$-mapping.
 \end{enumerate}
\end{Th}

Riley and Shalev determined when $u(L)$
is Lie nilpotent, solvable (for $p>2$), or satisfies the Engel condition~\cite{RiSh93}.

\begin{Th}[\cite{RiSh93}]\label{uLNS}
Let $u(L)$ be the restricted enveloping algebra of a restricted Lie algebra $L$ over a field $K$ of characteristic $p>0$.
\begin{enumerate}
  \item
    $u(L)$ is Lie nilpotent if and only if $L$ is nilpotent and $L^2$ is finite dimensional and $p$-nilpotent;
  \item
    $u(L)$ is $n$-Engel for some $n$ if and only if $L$ is nilpotent, $L^2$ is $p$-nilpotent, and $L$ has a restricted ideal $A$ such that both
    $L/A$ and $A^2$ are finite dimensional.
  \item
    $u(L)$ is solvable if and only if $L^2$ is finite dimensional and $p$-nilpotent, for $p\neq 2$.
    \end{enumerate}
\end{Th}

Let $L$ be a restricted Lie algebra, $\ch K=p>0$.
Similarly to the dimension subgroups, using upper Lie powers (see the previous section),
Riley and Shalev defined {\it Lie dimension subalgebras}~\cite{RiSh95}:
$$
D_{(n)}(L)=L\cap u(L)^{(n)}, \qquad n\ge 0.
$$
(Recall that our enumeration is shifted).
They also gave the following description~\cite{RiSh95}:
\begin{equation}\label{dim-subalg}
D_{(n)}(L)=\sum_{(i-1)p^k\ge n} \gamma_i(L)^{[p^k]},\quad n\ge 0.
\end{equation}
Siciliano proved~\cite{Sic06} that in case $p>2$,
the strong solvability of the restricted enveloping algebra $u(L)$ is equivalent to its solvability.
Moreover, the strong solvability in case $p=2$ is described by the same conditions of Part 3 of Theorem~\ref{uLNS}.
Also, in case $p=2$ he provided an example of the restricted enveloping algebra $u(L)$ that is
solvable but not strongly solvable.

The following is an analogue of results on the Lie nilpotency classes of group rings (Theorem~\ref{Tgr-nilp}).
\begin{Th}[\cite{RiSh95}]\label{Tu-nilp}
Let $L$ be a restricted Lie algebra over a field $K$ of characteristic $p>0$ such
that $u(L)$ is Lie nilpotent. Then
\begin{enumerate}
\item The strong Lie nilpotency class of $u(L)$ is equal to
$$ 1+(p-1)\sum_{m\ge 1}m \dim (D_{(m)}(L)/D_{(m+1)}(L)). $$
\item In case $p>3$, the Lie nilpotency class coincides with the strong Lie nilpotency class.
\end{enumerate}
\end{Th}

The solvability of restricted enveloping algebras in case of characteristic 2 was settled in~\cite{SiUs13}.
Lie nilpotence, solvability, and other non-matrix identities for (restricted) enveloping algebras of (restricted)
Lie (super)algebras are studied in~\cite{BRU,Pe92,Sic07,Sic11,SicSpi06,SiUs13.2,SiUs15.2,Use13,Use13.2}.
For other results on derived lengths, Lie nilpotency classes for $u(L)$, or identities for skew and symmetric elements
of $u(L)$, etc., see the survey~\cite{SiUs15}.

More general cases of (restricted) enveloping algebras for (color) Lie $p$-(super)algebras are treated in~\cite{BMPZ}.
Further developments have been obtained for smash products  $U(L)\#K[G]$ and $u(L)\#K[G]$,
where a group $G$ acts by automorphisms on a (restricted) Lie algebra $L$~\cite{BaPe02}.
Identities of smash products $U(L)\#K[G]$, where $L$ is a Lie superalgebra in characteristic zero were studied
by Kotchetov~\cite{Koc03}.
The Lie structure of smash products has been investigated in~\cite{SicUse17}.
The results on identities of smash products are of interest because they combine, as particular cases,
both, the results on identities of group ring and enveloping algebras.

\section{Poisson algebras and their identities}
\label{SPoisson}
\subsection{Poisson algebras}
Poisson algebras naturally appear in different areas of algebra, topology and physics.
Probably Poisson algebras were first introduced in 1967 by Berezin~\cite{Ber67}, see also Vergne~\cite{Ver69} (1969).
Poisson algebras are used to study universal enveloping algebras of
finite dimensional Lie algebras in characteristic zero~\cite{Kos81,Ooms12}.
In particular, abelian subalgebras in symmetric Poisson algebras are used to study
commutative subalgebras in universal enveloping algebras of finite-dimensional semisimple Lie algebras
in characteristic zero~\cite{Tar00,Vin90}.
Applying Poisson algebras, Shestakov and Umirbaev managed to solve a long-standing problem:
they proved that the Nagata automorphism of the polynomial ring in three variables $\mathbb{C}[x,y,z]$ is wild~\cite{SheUmi04}.
Related algebraic properties of free Poisson algebras were studied by
Makar-Limanov, Shestakov and Umirbaev~\cite{MakShe12,MakUmi11}.

The free Poisson algebras were defined by Shestakov~\cite{Shestakov93}.
A basic theory of identical relations for Poisson algebras  was developed by Farkas~\cite{Farkas98,Farkas99}.
See further developments on the theory of identical relations of Poisson algebras,
in particular, the theory of so called  codimension growth in characteristic zero
by Mishchenko, Petrogradsky, and Regev~\cite{MiPeRe}, and Ratseev~\cite{Rats14}.
\medskip

Recall that a vector space $A$ is a {\it Poisson algebra} provided that, beside the addition,
$A$ has two $K$-bilinear operations which are related by the Leibnitz rule. More precisely,
\begin{itemize}
\item
A is a commutative associative algebra with unit whose multiplication is denoted by $a\cdot b$ (or $ab$), where $a, b\in A$;
\item
$A$ is a Lie algebra whose product is traditionally denoted by the Poisson bracket $\{a, b\}$, where $a, b\in A$;
\item these two operations are related by the Leibnitz rule
\begin{equation*}
\{a\cdot b, c\}=a\cdot\{b, c\}+b\cdot\{a, c\},\qquad  a, b, c \in A.
\end{equation*}
\end{itemize}
\subsection{Examples of Poisson algebras}
Typical examples are as follows.
\medskip

\noindent {\bf Example 1.}\  Consider the polynomial ring
$\H_{2m}=K[X_1,\dots,X_m,Y_1,\dots,Y_m]$.
Set $\{X_i,Y_j\}=\delta_{i,j}$ and extend this bracket by the Leibnitz rule. We obtain the Poisson bracket:
$$\{f,g\}=\sum_{i=1}^m
\bigg(\frac{\partial f}{\partial X_i}\;\frac{\partial g}{\partial Y_i}-
\frac{\partial f}{\partial Y_i}\;\frac{\partial g}{\partial X_i}\bigg),
\qquad f,g\in \H_{2m}.
$$
The commutative product is the natural multiplication.
We obtain the {\it Hamiltonian} Poisson algebra~$\H_{2m}$.
\medskip

\noindent {\bf Example 2.}\
Let $L$ be a Lie algebra over an arbitrary field $K$,
$\{U_n\mid n\ge 0\}$ the natural filtration of its universal enveloping algebra $U(L)$.
Consider the {\it symmetric algebra}
$S(L)=\gr U(L)=\mathop{\oplus}\limits_{n=0}^\infty U_{n}/U_{n+1}$ (see~\cite{Dixmier}).
Recall that $S(L)$ is identified with the polynomial ring $K[v_i\,|\, i\in I]$, where $\{v_i\,|\, i\in I\}$ is a $K$-basis of $L$.
Define the Poisson bracket as follows.
Set $\{v_i,v_j\}=[v_i,v_j]$ for all $i,j\in I$,
and extend to the whole of $S(L)$ by linearity and using the Leibnitz rule. For example,
$$\{v_i\cdot v_j,v_k\}=v_i\cdot\{v_j,v_k\}+v_j\cdot\{v_i,v_k\},\qquad i,j,k\in I.$$
Thus, $S(L)$ has a structure of a Poisson algebra, called the {\it symmetric algebra} of $L$.
\medskip

\noindent {\bf Example 3.}\ Let $L$ be a Lie algebra with a $K$-basis $\{v_i\,|\, i\in I\}$, where $\ch K=p>0$.
Consider a factor algebra of the symmetric (Poisson) algebra
$$\mathbf{s}(L)=S(L)/(v^p\, |\, v\in L)\cong K[v_i\mid i\in I]/(v_i^p \,|\, i\in I),$$
we get an algebra of truncated polynomials. Observe that
$$\{v^p,u\}=pv^{p-1}\{v,u\}=0,\qquad v\in L,\ u\in \SSS(L).$$
So, the Poisson bracket on $S(L)$ yields a Poisson bracket on $\SSS(L)$.
Thus, $\SSS(L)$ is a Poisson algebra, we call it a {\it truncated symmetric algebra}.
Remark that the Lie algebra $L$ need not be restricted.

\medskip
\textbf{Example 4.}
Let $K$ be a field of positive characteristic $p$.
We introduce the {\it truncated Hamiltonian} Poisson  algebra as
$$\mathbf{h}_{2m}(K)=K[X_1,\dots,X_m,Y_1,\dots,Y_m]/(X_i^p,Y_i^p\mid  i=1,\dots,m),$$
where we define the bracket as in Example~1 using the observation of Example~3.
\medskip

The Hamiltonian algebras $\mathbf{h}_{2m}(K)$ and $\mathbf{H}_{2m}(K)$ in the class of Poisson algebras
play a role  similar to that of the matrix algebras $\mathrm{M}_n(K)$ for associative algebras.

\subsection{Poisson identities}
The objective of this subsection is to supply basic facts
on polynomial identities of Poisson algebras.

Consider the free Lie algebra $L=L(X)$ generated by a set $X$ and its symmetric algebra $F(X)=S(L(X))$.
Then, $F(X)$ is a {\it free Poisson algebra} in $X$, as was shown by Shestakov~\cite{Shestakov93}.
For example, let $L=L(x,y)$ be the free Lie algebra of rank 2. Consider its Hall basis~\cite{Ba}
$$L=\langle x, y, [y,x], [[y,x],x], [[y,x],y], [[[y,x],x],x],\ldots\rangle_K. $$
We obtain the free Poisson algebra $F(x,y)=S(L)$ of rank 2,
which has a canonical basis as follows:
$$
F(x,y)=\big\langle x^{n_1} y^{n_2} \{y,x\}^{n_3} \{\{y,x\},x\}^{n_4} \{\{y,x\},y\}^{n_5} \{\{\{y,x\},x\},x\}^{n_6}\cdots\,\Big|\, n_i\ge 0
\big\rangle_K,
$$
where only finitely many $n_i$, $i\ge 1$, are non-zero in the monomials above.
\medskip

A definition of  a {\it Poisson PI-algebra} is standard,
identities being elements of the free Poisson algebra  $F(X)$ of countable rank.
We assume that basic facts on identical relations of linear algebras are known to the reader
(see, e.g.,~\cite{Ba,Drensky}).
Farkas introduced so called {\it customary identities}~\cite{Farkas98}:
$$
\sum_{\substack{\sigma\in S_{2n}\\ \sigma(2k-1)<\sigma(2k),\ k=1,\dots,n\\\sigma(1)<\sigma(3)<\cdots<\sigma(2n-1)}}
\!\!\!\!\!\!\!\!\!\!\!\!\!
\mu_{\sigma} \{x_{\sigma (1)},x_{\sigma (2)}\}\cdots
\{x_{\sigma (2n-1)},x_{\sigma (2n)}\}\equiv 0,\quad \mu_\sigma\in K.
$$
where $\mu_e=1$, for the identity permutation.
Denote by $T_{2n}$ the set of permutations $\tau\in S_{2n}$ appearing in the some above.
The importance of customary identities is explained by the following fact.

\begin{Th}[\cite{Farkas98}]\label{TFarkas0}
Suppose that $\V$ is a nontrivial variety of Poisson algebras over
a field $K$ of characteristic zero.
Then $\V$ satisfies a nontrivial customary identity.
\end{Th}

Let us show the idea of the proof.
Let a Poisson algebra $R$ satisfy the identity $f(X,Y,Z)=\{\{X,Y\},Z\}\equiv 0$.
Then, $R$ also satisfies the identity:
\begin{align*}
0&\equiv f(X_1X_2,Y,Z)-X_1f(X_2,Y,Z)-X_2f(X_1,Y,Z)\\
&=\{\{X_1X_2,Y\},Z\}- X_1\{\{X_2,Y\},Z\}- X_2\{\{X_1,Y\},Z\}\\
&=\{X_1,Y\}\{X_2,Z\}+\{X_1,Z\}\{X_2,Y\},
\end{align*}
which is customary.
Farkas called this process a {\it customarization}~\cite{Farkas98}, it is an analogue
of the linearization process for associative algebras.
The arguments of \cite{Farkas98} actually prove the following.
\begin{Th}[\cite{Farkas98}]\label{TFarkas}
Suppose that a Poisson algebra $A$ over an arbitrary field
satisfies a nontrivial {\bf multilinear} Poisson identity.
Then $A$ satisfies a nontrivial customary identity.
\end{Th}

\subsection*{Remark}
Let us explain why we {\it need all polynomials to be multilinear in case of positive characteristic} $p$.
The linearization process is simply not working for Poisson algebras in positive characteristic
as it does for associative and Lie algebras.
For example, the Poisson identity $\{x,y\}^p\equiv 0$ is given by a nonzero element of
the free Poisson algebra $F(X)$.
Observe that its full linearization is trivial:
$$\sum_{\sigma,\pi\in S_p }\{x_{\sigma(1)},y_{\pi(1)}\}\cdots \{x_{\sigma(p)},y_{\pi(p)}\}=
p!\sum_{\pi\in S_p }\{x_{1},y_{\pi(1)}\}\cdots \{x_{p},y_{\pi(p)}\}=0.
$$
Moreover, let us check that any truncated symmetric algebra $\SSS (L)$
satisfies the identity $\{x,y\}^p\equiv 0$. Indeed, let $a,b\in
\SSS (L)$, then $\{a,b\}$ is a truncated polynomial without constant
term, its $p$th power is zero by the Frobenius rule
$(v+w)^p=v^p+w^p$. Thus, it does not make sense to study nonlinear
Poisson identities for truncated symmetric algebras.
\medskip

In the theory of Poisson PI-algebras, the analogue of the standard polynomial is (\cite{Farkas98},~\cite{Farkas99}):
$$ \St_{2n}=\St_{2n}(x_1,\dots,x_{2n})=\sum_{\sigma\in T_{2n}}(-1)^\sigma \{x_{\sigma (1)},x_{\sigma (2)}\}\cdots
\{x_{\sigma (2n-1)},x_{\sigma (2n)}\}. $$
This is a customary polynomial, skewsymmetric in all variables \cite{Farkas98}.
One has the following fact similar to the theory of associative algebras.

\begin{Th}[\cite{MiPeRe}]
In case of zero characteristic, any Poisson PI-algebra satisfies an identity $(\St_{2n})^m\equiv 0$, for some integers $n,m$.
\end{Th}
This result was proved by establishing an analogue of Regev's theorem on codimension growth.
Moreover, it was proved that so called {\it customary codimension growth} is exponential
with an integer exponent~\cite{MiPeRe}.

Another important fact on the standard identity is as follows.
\begin{Lemma}[\cite{Farkas98}]
Let $A$ be a Poisson algebra over an arbitrary field $K$ and $A$ is $k$-generated as an associative algebra.
Then it satisfies the standard identity
$\St_{2m}\equiv 0$,
whenever $2m>k$.
\end{Lemma}

\section{Multilinear identities of symmetric Poisson algebras}

The following result is an analogue of the classical Amitsur-Levitzki theorem on identities of matrix algebras.
Kostant used another terminology,
but as observed by Farkas~\cite{Farkas99}, this is a result on identities of symmetric Poisson algebras.

\begin{Th} [\cite{Kos81,Farkas99}]\label{TKost}
Let $L$ be a finite dimensional Lie algebra over a field of characteristic zero.
The symmetric algebra $S(L)$ satisfies the standard Poisson identity $\St_ {2d}\equiv 0$
as soon as $2d$ exceeds the dimension of a maximal coadjoint orbit of $L$.
\end{Th}

The Lie nilpotence of class 2 of symmetric algebras $S(L)$, where $L$ is a Lie superalgebra,
was characterized by Shestakov. The next statement follows from Theorem 4 and Theorem 5 of~\cite{Shestakov93}.
\begin{Th}[\cite{Shestakov93}]\label{TSh93}
The symmetric algebra $S(L)$ of a Lie algebra $L$ over a field $K$ satisfies
the identity $\{x,\{y,z\}\}\equiv 0$ if and only if $L$ is abelian.
\end{Th}

Farkas proved the following statement that generalizes Kostant's Theorem~\ref{TKost}.

\begin{Th}[\cite{Farkas99}] \label{TSymm0}
Let $L$ be a Lie algebra over a field of characteristic zero.
Then the symmetric algebra $S(L)$ satisfies a
nontrivial Poisson identity if and only if $L$ contains an abelian
subalgebra of finite codimension.
\end{Th}
Giambruno and Petrogradsky extended this result to an arbitrary characteristic~\cite{GiPe06}.

\begin{Th}[\cite{GiPe06}]\label{TSymm_p}
Let $L$ be a Lie algebra over an arbitrary field.
Then the symmetric algebra $S(L)$ satisfies a nontrivial {\em multilinear}
Poisson identity if and only if $L$ contains an abelian subalgebra of finite codimension.
\end{Th}

The following result was obtained for the truncated symmetric algebra $\SSS(L)$
of a restricted Lie algebra $L$.
\begin{Th}[\cite{GiPe06}]\label{Tmain}
Let $L$ be a restricted Lie algebra.
Then the truncated symmetric algebra
$\SSS(L)$ satisfies a nontrivial {\em multilinear} Poisson identity if and only if
there exists a restricted ideal $H\subseteq L$ such that
 \begin{enumerate}
   \item
     $\dim L/H<\infty$;
   \item
     $\dim H^2<\infty$;
   \item
     $H$ is nilpotent of class 2.
 \end{enumerate}
\end{Th}

\section{Lie identities of symmetric Poisson algebras}
\label{Smain}
Now we discuss special cases of identities of symmetric Poisson algebras.
Remark that the identities of the (strong) Lie nilpotence and (strong) solvability are multilinear,
thus Theorem~\ref{Tmain} can be applied for such algebras.

\subsection{Lie nilpotence of truncated symmetric algebras $\SSS(L)$}
Let $R$ be a Poisson algebra.
Consider the {\em lower central series} of $R$ as a Lie algebra,
i.e., $\gamma_1(R)=R$ and $\gamma_{n+1}(R)=\{\gamma_n(R),R\}, n\geq1$.
We say that $R$ is {\it Lie nilpotent of class $s$} if and only if  $\gamma_{s+1}(R)=0$ but $\gamma_s(R)\neq 0$.
Clearly, the condition $\gamma_{s+1}(R)=0$ is equivalent to the {\it identity of Lie nilpotence of class} $s$:
$$\{\ldots\{\{X_0,X_1\},X_2\},\ldots,X_s\}\equiv 0.$$

Similarly to the associative case, one defines {\em upper Lie powers}.
At each step we take the ideal generated by commutators, namely,
put $R^{(0)}=R$ and $R^{(n)}=\{R^{(n-1)},R\}\cdot R$ for all $n\geq 1$
(the enumeration is shifted, because $\{R^{(n)}| n\ge 0\}$ is a filtration, see also~\cite{Pe11}).
A Poisson algebra $R$ is {\em strongly Lie nilpotent of class s}
iff $R^{(s)}= 0$ but $R^{(s-1)}\ne 0$.
The condition $R^{(s)}=0$ is equivalent to the identical relation of the {\em strong Lie nilpotence of class~$s$}:
\begin{equation*}
\{\{\ldots\{\{X_0,X_1\}\cdot Y_1,X_2\}\cdot Y_2,\ldots, X_{s-1}\}\cdot Y_{s-1},X_s\}\equiv 0.
\end{equation*}
 Observe that
\begin{equation}
\gamma_n(R)\subseteq R^{(n-1)},\ n\geq1.\label{gamma(R)em R(n-1)}
\end{equation}
So, the strong Lie nilpotence of class $s$ implies the Lie nilpotence of class at most $s$.
The Lie nilpotence of class $1$ is equivalent to the strong Lie nilpotence of class $1$ and equivalent
that $R$ is abelian.

\begin{Th}[\cite{PeIlana17}]\label{Tnilp}
Let $L$ be a Lie algebra over a field of positive characteristic $p$. Consider
its truncated symmetric Poisson algebra $\SSS(L)$. The following conditions are equivalent:
\begin{enumerate}
\item $\SSS(L)$ is strongly Lie nilpotent;
\item $\SSS(L)$ is Lie nilpotent;
\item $L$ is nilpotent and $\dim L^2<\infty$.
\end{enumerate}
\end{Th}
Let $L$ be a Lie algebra over a field $K$ ($\ch K=p>0$).
Using the upper Lie powers, define
the {\it Poisson dimension subalgebras} ({\it truncated Poisson dimension subalgebras}, respectively)  of $L$ as:
\begin{alignat*}{2}
D^{S}_{(n)}(L)&=L\cap (S (L))^{(n)},  &\qquad n&\ge 0;\\
D^{s}_{(n)}(L)&=L\cap (\SSS (L))^{(n)},& \qquad n&\ge 0.
\end{alignat*}
We obtain a description of these subalgebras~\cite{PeIlana17} similar to that
for group rings~\eqref{dim-subgroups} and restricted enveloping algebras~\eqref{dim-subalg}
(compare with the first term of that products):
\begin{alignat*}{2}
D^S_{(n)}(L)&=\gamma_{n+1}(L),&\qquad n&\ge 0;\\
D^s_{(n)}(L)&=\gamma_{n+1}(L),&\qquad n&\ge 0.
\end{alignat*}
We compute the classes of Lie nilpotence and strong Lie nilpotence.
We get an analogue of the formulas known for group rings (Theorem~\ref{Tgr-nilp}) and restricted enveloping algebras
(Theorem~\ref{Tu-nilp}). The analogy is better seen in terms of truncated Poisson dimension subalgebras.

\begin{Th}[\cite{PeIlana17}]\label{Tclasses}
Let $L$ be a Lie algebra over a field of positive characteristic $p>3$,
such that the truncated symmetric Poisson algebra $\SSS(L)$ is Lie nilpotent.
The following numbers are equal:
\begin{enumerate}
\item the strong Lie nilpotency class of $\SSS(L)$;
\item the Lie nilpotency class of  $\SSS(L)$;
\item $$1+(p-1)\sum_{n\geq1}n\cdot\dim(\gamma_{n+1}(L)/\gamma_{n+2}(L)).$$
\end{enumerate}
In cases $p=2,3$, the numbers 1) and 3) remain equal.
\end{Th}
In case $p=2,3$, the number above yields an upper bound for the Lie nilpotency class.
Also, we have a lower bound for the Lie nilpotency class, $L$ being non-abelian~\cite{PeIlana17}:
$$2+(p-1)\sum_{n\geq2}(n-1)\cdot\dim(\gamma_{n+1}(L)/\gamma_{n+2}(L)).$$

\subsection{Solvability of truncated symmetric algebras $\SSS(L)$}
Let $R$ be a  Poisson algebra.
Consider its {\it derived series} as a Lie algebra:
$\delta_0(R)=R$, $\delta_{n+1}(R)=\{\delta_n(R),\delta_n(R)\}$, $n\geq 0$.
Polynomials of {\it solvability} are defined as: $\delta_1(X_1,X_2)=\{X_1,X_2\}$ and
\begin{equation*}
\delta_{n+1}(X_1,X_2,\ldots,X_{2^{n+1}})
=\big\{\delta_{n}(X_1,\ldots,X_{2^n}),\delta_n(X_{2^{n}+1},\ldots,X_{2^{n+1}})\big\}, \quad n\ge 1.
\end{equation*}
A Poisson algebra $R$ is \textit{solvable of length $s$} if,
and only if, $\delta_{s}(R)=0$ and $\delta_{s-1}(R)\ne 0$,
or equivalently, $R$ satisfies the above identity of Lie solvability
$\delta_s(\ldots)\equiv 0$, $s$ being minimal.

Define the {\it upper derived series}:
$\tilde{\delta}_0(R)=R$ and $\tilde{\delta}_{n+1}(R)=\{\tilde{\delta}_{n}(R),\tilde{\delta}_{n}(R)\}\cdot R$, $n\geq 0$.
Define polynomials of the {\it strong solvability} by
$\tilde{\delta}_1(X_1,X_2,Y_1)=\{X_1,X_2\}\cdot Y_1$, and
\begin{multline*}
\tilde{\delta}_{n+1}(X_1,\ldots,X_{2^{n+1}},Y_1,\ldots,Y_{2^{n+1}-1}) \\
=\Big\{\tilde{\delta}_{n}(X_1,\ldots,X_{2^{n}},Y_1,\ldots,Y_{2^{n}-1}),
 \tilde{\delta}_{n}(X_{2^{n}+1},\ldots,X_{2^{n+1}},Y_{2^{n}},\ldots,Y_{2^{n+1}-2})\Big\} \cdot Y_{2^{n+1}-1},
 \quad n\ge 1.
\end{multline*}

A Poisson algebra $R$ is {\it strongly solvable of length} $s$ iff
$\tilde{\delta}_{s}(R)=0$ and $\tilde{\delta}_{s-1}(R)\ne 0$,
or equivalently $R$ satisfies $\tilde\delta_s(\ldots)\equiv 0$, $s$ being minimal.
Observe that
\begin{equation}\label{deltindeltb}
\delta_{s}(R)\subseteq\tilde{\delta}_{s}(R),\qquad s\geq 0.
\end{equation}
So, strong solvability of length $s$ implies solvability of length at most $s$.
The solvability of length $1$ is equivalent to the strong solvability of length $1$ and equivalent
that $R$ is an abelian Lie algebra.

\begin{Th}[\cite{PeIlana17}]\label{Tsolv}
Let $L$ be a Lie algebra over a field of positive characteristic $p\geq 3$. Consider
its truncated symmetric Poisson algebra $\SSS(L)$. The following conditions are equivalent:
\begin{enumerate}
\item $\SSS(L)$ is strongly solvable;
\item $\SSS(L)$ is solvable;
\item $L$ is solvable and $\dim L^2<\infty$.
\end{enumerate}
In case $p=2$, conditions 1) and 3) remain equivalent.
\end{Th}

Solvability length of symmetric Poisson algebras was further studied in~\cite{Sic20}.
Namely, an upper and lower bounds for the strong derived length of $s(L)$ are obtained.
It is established when $s(L)$ is metabelian, that is, it has derived length 2.
For a non-abelian Lie algebra, a lower bound for the derived length of $s(L)$ is obtained.
Finally, necessary and sufficient conditions under which that value is attained are determined.

Observe that the description of solvable group rings in characteristic 2
looks very nice (Theorem~\ref{TKGNS}).
But the answer to a similar question for the restricted enveloping algebras is rather complicated
and was obtained only recently in~\cite{SiUs13}.

The problem of solvability of $\SSS(L)$ in case $\ch K=2$ is open.
We show that the situation is different from other characteristics.
Namely, in case $\ch K=2$, we give two examples of truncated symmetric Poisson algebras that are solvable
but not strongly solvable, see Lemma~\ref{L22} and Lemma~\ref{L22b}.
A close fact is that the Hamiltonian algebras $\mathbf{H}_2(K)$ and $\mathbf{h}_2(K)$ are solvable
but not strongly solvable in case $\ch K=2$ (Lemma~\ref{L23}).
This is an analogue of a well-known fact that the matrix ring $\mathrm{M}_2(K)$ of $2\times 2$ matrices over
a field $K$, $\ch K=2$, is solvable but not strongly solvable.

\subsection{Nilpotency and solvability of symmetric algebras $S(L)$}
The following extension of a result of Shestakov~\cite{Shestakov93} is proved.

\begin{Th}[\cite{PeIlana17}]\label{Pe16}
Let $L$ be a Lie algebra over a field $K$, and $S(L)$ its symmetric Poisson algebra.
The following conditions are equivalent:
\begin{enumerate}
\item $L$ is abelian;
\item $S(L)$ is strongly Lie nilpotent;
\item $S(L)$ is Lie nilpotent;
\item $S(L)$ is strongly solvable;
\item $S(L)$ is solvable (here assume that $\ch K\ne 2$).
\end{enumerate}
\end{Th}
In case $\ch K=2$, the solvability of the symmetric Poisson algebra $S(L)$ is an open question.
Two examples of Lie algebras mentioned above also
yield solvable symmetric algebras which are not strongly solvable (Lemma~\ref{LS2solvA} and Lemma~\ref{LS2solvB}).
\subsection*{Remark}
Formally, our statements on ordinary Lie nilpotency and solvability are concerned
only with the {\it Lie structure} of the Poisson algebras $\SSS(L)$ and $S(L)$.
But our proof heavily relies on Theorem~\ref{Treduction} of~\cite{GiPe06}, which in turn
uses the existence of a nontrivial customary identity given by Theorem~\ref{TFarkas} (Farkas~\cite{Farkas98}).
In this way, we need the {\it Poisson structure} of our algebras to prove our results.
We do not see ways to prove them using the theory of Lie identical relations only.

\subsection{Delta-sets and multilinear Poisson identical relations}

Now we present an important instrument to prove our results  on Poisson identities in enveloping algebras.
{\it Delta-sets} in groups were introduced by Passman to study identities in the group rings~\cite{Pas72}.
Namely, let $G$ be a group, then
\begin{align*}
\Delta_n(G)&:=\{a\in G\mid |a^G|\leq n\}, \qquad n\geq0;\\
\Delta(G) &:=\mathop{\cup}\limits_{n=0}^{\infty}\Delta_n(G)=\{a\in G\mid |a^G|<\infty\}.
\end{align*}
A crucial step to specify group ring with identical relations was to establish that
there exist integers $n,m$ such that $|G:\Delta_n(G)|\le m$, see~\cite{Pas72}.

In case of Lie algebras, the delta-sets were introduced by Bahturin to study identical relations of the universal enveloping algebras~\cite{Ba74}.
Let $L$ be a Lie algebra, one defines the delta-sets as sets of elements of {\it finite width} as follows:
\begin{align*}
\Delta_n(L)&:=\{x\in L\mid\dim [L,x]\leq n\}, \qquad n\geq0;\\
\Delta(L) &:=\mathop{\cup}\limits_{n=0}^{\infty}\Delta_n(L)=\{x\in L\mid\dim [L,x]<\infty\}.
\end{align*}

Note that $\Delta_n(L)$, $n\ge 0$, is not a subalgebra or even a subspace in a general case.
The basic properties of the delta-sets are as follows.

\begin{Lemma}[\cite{BMPZ,Pe92}]
\label{PropDel}
Let $L$ be a (restricted) Lie algebra, $n,m\ge 0$.
 \begin{enumerate}
  \item
    $\Delta_n(L)$ is invariant under scalar multiplication;
  \item
    if $x\in\Delta_n(L)$, $y\in\Delta_m(L)$, then
     $\alpha x+\beta y\in\Delta_{n+m}(L)$, where  $\alpha,\beta\in K$;
  \item
    if $x\in\Delta_n(L)$, $y\in L$, then
     $[x,y]\in\Delta_{2n}(L)$;
  \item
    if $x\in\Delta_n(L)$ and $L$ a restricted
    Lie algebra, then $x^{[p]}\in\Delta_{n}(L)$;
  \item
    $\Delta(L)$  is a (restricted) ideal of $L$.
\end{enumerate}
\end{Lemma}
\begin{Lemma}[\cite{RiSh93}]\label{DeltaI}
Let $L$ be a Lie algebra.
\begin{enumerate}
\item
     if $I$ is a finite dimensional ideal of $L$, then $\Delta(L/I)=(\Delta(L)+I)/I$;
\item
     if $H$ is a subalgebra of finite codimension in $L$, then $\Delta(H)=\Delta(L)\cap H$.
\end{enumerate}
\end{Lemma}

Suppose that $W$ is a subset in a $K$-vector space $V$.
We say that $W$ has finite codimension in $V$ if there exist
$v_1,\dots,v_m\in V$ such that
$V=\{w+\lambda_1v_1+\dots+\lambda_m v_m \,|\,w\in W,\ \lambda_1,\ldots, \lambda_m\in K\}$.
If $m$ is the minimum integer with such property, then we write $\dim V/W=m$.
We also introduce the notation
$m\cdot W=\{w_1+\cdots+w_m\,|\, w_i\in W\}$, where $m\in\N$.
\begin{Lemma}[{\cite[Lemma~6.3]{BaPe02}}]
\label{Ldilat}
Let $V$ be a $K$-vector space. Suppose that a subset
$T\subseteq V$ is stable under multiplication by scalars and $\dim V/T\le n$.
Then the linear span is obtained as: $\langle T\rangle_K=4^n\cdot T$.
\end{Lemma}

We need a result on bilinear maps.
\begin{Th}[P.M. Neumann,~\cite{Ba}]
\label{TNeu}
Let $U,V,W$ be vector spaces over a field $K$ and
$\varphi:U\times V \to W$ a bilinear map. Suppose that for all
$u\in U$ and $v\in  V$, $\dim \varphi (u,V)\le m$ and
$\dim \varphi (U,v)\le l$. Then $\dim \langle\varphi (U,V)\rangle_K\leq ml$.
\end{Th}

The following facts were indispensable in our approach to study Poisson identical
relations in symmetric algebras of Lie algebras.
Actually the following result was proved for restricted Lie algebras~\cite{GiPe06},
but its proof remains valid for truncated symmetric algebras as well.
\begin{Th}[\cite{GiPe06}]
\label{TtrPoisson}
Let $L$ be a Lie algebra.
Suppose that the symmetric algebra $S(L)$
(or the truncated symmetric algebra $\SSS(L)$) satisfies a multilinear Poisson identity.
Then there exist integers $n,N$ such that $\dim L/\Delta_N(L)<n$.
\end{Th}
It yields the following reduction step, which is actually contained in~\cite{GiPe06}.
\begin{Th}[\cite{GiPe06}]
\label{Treduction}
Let $L$ be a Lie algebra such that
the symmetric algebra $S(L)$
(or the truncated symmetric algebra $\SSS(L)$) satisfies a multilinear Poisson identity.
Let $\Delta=\Delta(L)$.
Then there exist integers $n,M$ such that
\begin{enumerate}
\item $\Delta=\Delta_M(L)$;
\item $\dim L/\Delta<n$;
\item $\dim \Delta^2\le M^2 $.
\end{enumerate}
\end{Th}

\subsection{Products of commutators in Poisson algebras}\label{Scomm}
Now, we supply technical results on products of commutators in Poisson algebras  that were used to get a lower bound on
the Lie nilpotency class of $\SSS(L)$.

Products of terms of the lower central series for {\it associative algebras} appear in works of many mathematicians,
the results being reproved without knowing the earlier works.
We do not pretend to make a complete survey here.
Probably, the first observations on products of commutators in associative algebras were
made by Latyshev in 1965~\cite{Lat65} and Volichenko in 1978~\cite{Vol78}.
There are further works, see e.g.~\cite{BapJor13,EtiKimMa09,Gor07,GriPche15,Kras13}.

In case of associative algebras,
Claim~1 of Theorem~\ref{Tcomutt}, probably, first was established by
Sharma-Shrivastava in 1990,~\cite[Theorem~2.8]{ShaSri90}.
As was remarked in~\cite{RiSh95},
the proof of the associative version of Claim~2 of Theorem~\ref{Tcomutt} is implicitly contained in~\cite{ShaSri90},
where it is proved for group rings.
A weaker statement (the associative version of~Lemma~\ref{Lgamma-2}) is established by
Gupta and Levin~\cite[Theorem~3.2]{GupLev83}.

The following statement is a Poisson version of respective results for associative algebras.
The validity of it is not automatically clear and it was checked directly,
following a neat approach due to Krasilnikov~\cite{Kras13}.

\begin{Th}[\cite{PeIlana17}]\label{Tcomutt}
Let $R$ be a Poisson algebra over a field $K$, $\ch K\neq 2,3$.
\begin{enumerate}
\item Suppose that one of integers $n,m\geq 1$ is odd, then
     $$\gamma_n(R)\cdot\gamma_m(R)\subseteq\gamma_{n+m-1}(R) R.$$
\item For all $x_1,\ldots,x_n\in R,$ $n,m\geq 1$, we have
     $$\{x_1,\ldots,x_n\}^m\in\gamma_{(n-1)m+1}(R) R.$$
\end{enumerate}
\end{Th}

The following is an analogue of a result for associative algebras, see~\cite[Theorem~3.2]{GupLev83}.
It is weaker than Claim~1 of Theorem~\ref{Tcomutt}, but it is valid for an arbitrary characteristic.
\begin{Lemma}[\cite{PeIlana17}]\label{Lgamma-2}
Let $R$ be a Poisson algebra over arbitrary field $K$. Then
$$
\gamma_m(R)\gamma_n(R)\subseteq \gamma_{m+n-2}(R)R, \qquad n,m\ge 2.
$$
\end{Lemma}
\subsection{Solvability of symmetric algebras $\SSS(L)$ and $S(L)$ in case $\ch K=2$}
We supply two examples of truncated symmetric algebras that are solvable but not strongly solvable.
\begin{Lemma}[\cite{PeIlana17}]\label{L22}
Let $L=\langle x, y_i \,|\,  [x,y_i]=y_i,\, i\in\N\rangle_K$, $\ch K=2$,
the remaining commutators being trivial. Then
\begin{enumerate}
\item $L^2=\Delta(L)=\Delta_1(L)=\langle y_i\,|\,  i\in\N\rangle$;
\item $\SSS(L)$ is solvable of length 3;
\item $\SSS(L)$ is not strongly solvable.
\end{enumerate}
\end{Lemma}

\begin{Lemma}[\cite{PeIlana17}]\label{L22b}
Let
$L=\langle x, y_i,z_i \,|\,  [x,y_i]=z_i,\, i\in\N\rangle_K$, $\ch K=2$,
the remaining commutators being trivial.
Then
\begin{enumerate}
\item $\Delta(L)=\Delta_1(L)=\langle y_i, z_i\,|\,  i\in\N\rangle$
      and $L^2=\langle z_i\,|\,  i\in\N\rangle$;
\item $\SSS(L)$ is solvable of length 3;
\item $\SSS(L)$ is not strongly solvable.
\end{enumerate}
\end{Lemma}

Two examples above are closely related to the following observation.
\begin{Lemma}[\cite{PeIlana17}]\label{L23}
Consider the truncated Hamiltonian Poisson algebra $P=\mathbf{h}_2(K)$
(or the Hamiltonian  Poisson algebra $P=\mathbf{H}_2(K)$),  $\ch K=2$.
Then
\begin{enumerate}
\item $P$ is solvable of length 3.
\item $P$ is not strongly solvable.
\end{enumerate}
\end{Lemma}
\begin{proof}
Let $P=\mathbf{h}_2(K)=K[X,Y]/(X^2,Y^2)=\langle 1,x,y,xy\rangle _K$, where $x,y$ denote the images of $X,Y$.
We have
$\delta_1(P)=\{P,P \}=\langle 1,x,y\rangle _K$,
$\delta_2(P)=\langle 1\rangle _K$, and
$\delta_3(P)=0$.
Also, one checks that $P$ is not strongly solvable.

Let $P=\mathbf{H}_2(K)=K[X,Y]$.
The Poisson brackets of monomials $X^nY^m$, $n,m\ge 0$ depend on parities of $n,m$ of multiplicands.
For simplicity, denote by $X^{\bar 0}Y^{\bar 1}$ all monomials $X^\a Y^\b\in K[X,Y]$ such that $\a$ is even and $\b$ odd, etc.
We get non-zero products only in the cases:
\begin{align*}
\{X^{\bar 1}Y^{\bar 0},X^{\bar 0}Y^{\bar 1}\} &= X^{\bar 0}Y^{\bar 0};\\
\{X^{\bar 1}Y^{\bar 1},X^{\bar 1}Y^{\bar 0}\} &= X^{\bar 1}Y^{\bar 0};\\
\{X^{\bar 1}Y^{\bar 1},X^{\bar 0}Y^{\bar 1}\} &= X^{\bar 0}Y^{\bar 1}.
\end{align*}
Thus, $\delta_1(P)$ is spanned by monomials of three types obtained above.
Consider their commutators, the first line yields that $\delta_2(P)$ is spanned by
monomials of type $Y^{\bar 0}Y^{\bar 0}$. Finally, $\delta_3(P)=0$.
\end{proof}
Thus, the Poisson algebras $\mathbf{h}_2(K)$,
$\mathbf{H}_2(K)$ in characteristic 2 behave similarly to the associative algebra $\mathrm{M}_2(K)$ of $2\times 2$ matrices
in characteristic 2.

The question of the solvability of the symmetric algebra  $S(L)$ in case $\ch K=2$
is more complicated as shown below.
The algebras of Lemma~\ref{L22} and Lemma~\ref{L22b},
also yield solvable symmetric algebras which,  of course, are not strongly solvable by that Lemmas.

\begin{Lemma}[\cite{PeIlana17}]\label{LS2solvA}
Let $L=\langle x,y_i\,|\,[x,y_i]=y_i,i\in\N\rangle_K$,
the other commutators being trivial, $\ch K=2$.
Then the symmetric Poisson algebra $S(L)$ is solvable of length 3 but not strongly solvable.
\end{Lemma}
\begin{proof}
Put $H=\langle y_i\,|\,i\in\N\rangle_K$.
For a monomial $v=y_{i_1}y_{i_2}\cdots y_{i_k}\in S(H)$ we define its length $|v|=k$.
Then $v'=\{x,v\}=|v|v$.
A basis of $S(L)$ is formed by
$x^\a v$, $\a\ge 0$, where $v\in S(H)$ are respective basis monomials.
Consider the products:
\begin{equation} \label{xaaxbb2}
\{ x^\a v,x^\b w\} =x^{\a+\b-1}(\a|w|+ \beta|v|) vw.
\end{equation}
These products depend on the parities of $\a,\b,|v|,|w|$.
For simplicity, denote by $x^{\bar 0}v^{\bar 1}$ all monomials $x^\a v\in S(L)$ such that $\a$ is even and $|v|$ is odd, etc.
The only non-zero products~\eqref{xaaxbb2} are of types:
\begin{align*}
\{x^{\bar 1}v^{\bar 0},x^{\bar 0}v^{\bar 1}\} &= x^{\bar 0}v^{\bar 1};\\
\{x^{\bar 1}v^{\bar 1},x^{\bar 1}v^{\bar 0}\} &= x^{\bar 1}v^{\bar 1};\\
\{x^{\bar 1}v^{\bar 1},x^{\bar 0}v^{\bar 1}\} &= x^{\bar 0}v^{\bar 0}.
\end{align*}
Thus, $\delta_1(S(L))$ is spanned by monomials of the three types obtained above.
Consider their commutators, the last line yields that $\delta_2(S(L))$ is spanned by
monomials of type $x^{\bar 0}v^{\bar 0}$. Finally, $\delta_3(S(L))=0$.
\end{proof}

\begin{Lemma}[\cite{PeIlana17}]\label{LS2solvB}
Let $L=\langle x,y_i,z_i\,|\,[x,y_i]=z_i,i\in\N\rangle_K$, the other commutators being trivial, $\ch K=2$.
Then the symmetric Poisson algebra $S(L)$ is solvable of length 3 but not strongly solvable.
\end{Lemma}


\end{document}